\documentclass[a4paper,10pt]{article}
\usepackage[utf8]{inputenc}

\usepackage{cmap}
\usepackage[T1]{fontenc}

\usepackage[english]{babel}
\usepackage{amsfonts}
\usepackage{amssymb, amsthm}
\usepackage{xcolor,graphicx}
\usepackage{marginnote}
\usepackage{amsmath}
\usepackage{a4wide}
\usepackage[affil-it]{authblk}

\usepackage{epstopdf}
\usepackage{titlesec}
\titlelabel{\thetitle.\quad}

\usepackage[labelsep=period]{caption}

\usepackage{enumitem}

\numberwithin{equation}{section}

\let\OLDthebibliography\thebibliography
\renewcommand\thebibliography[1]{
	\OLDthebibliography{#1}
	\setlength{\parskip}{1pt}
	\setlength{\itemsep}{1pt plus 0.3ex}
}

\usepackage[colorlinks=true,linktocpage,pdfpagelabels,
bookmarksnumbered,bookmarksopen]{hyperref}
\definecolor{ForestGreen}{rgb}{0.1,0.6,0.05}
\definecolor{EgyptBlue}{rgb}{0.063,0.1,0.6}
\hypersetup{
	colorlinks=true,
	linkcolor=EgyptBlue,         
	citecolor=ForestGreen,
	urlcolor=olive
}

\usepackage[hyperpageref]{backref}

\usepackage{accents}
\def\reals{\mathbb R}
\def\wolp{\accentset{\circ}{W}_p^{1}}
\def\wlp{W_p^{1}}

\newtheorem{thm}{Theorem}[section]
\newtheorem{lemma}[thm]{Lemma}

\newtheorem{cor}[thm]{Corollary}
\theoremstyle{definition}
\newtheorem{remark}[thm]{Remark}

\title{On a property of the nodal set of least energy sign-changing solutions for quasilinear elliptic equations
	\\ \medskip}

\author[1]{Vladimir Bobkov\thanks{E-mail: \texttt{bobkov@kma.zcu.cz}}}
\affil[1]{{\small Department of Mathematics and NTIS, Faculty of Applied Sciences, University of West Bohemia, Univerzitn\'i 8, 306 14 Plze\v{n}, Czech Republic}}
\author[2]{Sergey Kolonitskii\thanks{E-mail: \texttt{s.kolonitsky@spbu.ru}}}
\affil[2]{{\small St.~Petersburg State University, 7/9 Universitetskaya nab., St.~Petersburg, 199034 Russia}}
\date{}

\begin{document}
 \maketitle

\begin{abstract}
	In this note we prove the Payne-type conjecture about the behaviour of the nodal set of least energy sign-changing solutions for the equation $-\Delta_p u = f(u)$ in bounded Steiner symmetric domains $\Omega \subset \mathbb{R}^N$ under the zero Dirichlet boundary conditions.
	The nonlinearity $f$ is assumed to be either superlinear or resonant. In the latter case, least energy sign-changing solutions are second eigenfunctions of the zero Dirichlet $p$-Laplacian in $\Omega$.
	We show that the nodal set of any least energy sign-changing solution intersects the boundary of $\Omega$. The proof is based on a moving polarization argument.
	
	
\end{abstract}
 
\section{Introduction}
Consider the boundary value problem
\begin{equation}\label{D}
\tag{$\mathcal{D}$}
\left\{
\begin{aligned}
-\Delta_p u &= f(u) 
&&\text{in } \Omega, \\
u &= 0  &&\text{on } \partial \Omega,
\end{aligned}
\right.
\end{equation}
where $\Delta_p u := \text{div}(|\nabla u|^{p-2} \nabla u)$ is the $p$-Laplacian, $p>1$. 
We impose the following assumptions on $\Omega$:
\begin{itemize}\addtolength{\itemindent}{1em}
	\item[$(O_1)$] $\Omega \subset \mathbb{R}^N$ is a bounded connected open set, $N \geqslant 2$.
	\item[$(O_2)$]
	$\Omega$ is Steiner symmetric with respect to the hyperplane $H_0 := \{(x_1,\dots,x_N) \in \mathbb{R}^N: x_1 = 0\}$, see \cite[p.~14]{kawohl}.
	Equivalently, $\Omega$ is convex with respect to the axis $e_1$ (i.e., any line segment parallel to $e_1$ with endpoints in $\Omega$ is contained in $\Omega$) and symmetric with respect to $H_0$. 
	\item[$(O_3)$] If $\Omega$ satisfies the interior ball condition at $x \in \partial \Omega$, then $\partial \Omega$ is of class $C^{1,\varsigma}$ in a neighbourhood of $x$ for some $\varsigma \in (0,1)$.
	\item[$(O_4)$] $\Omega$ satisfies the exterior ball condition at any $x \in \partial \Omega$.
\end{itemize}

For instance, an open ball centred at the origin or a convex polytope which is symmetric with respect to $H_0$ satisfies $(O_1)-(O_4)$, but a spherical shell does not satisfy $(O_2)$. 
While $(O_1)$ and $(O_2)$ are principal for our arguments, the assumptions $(O_3)$ and $(O_4)$ can potentially be relaxed. 

\medskip
We will consider two main types of the nonlinearity $f: \mathbb{R} \to \mathbb{R}$:
\begin{enumerate}[leftmargin=0.5cm]
	\item[I.] The first type contains $f$ with \textit{superlinear} and subcritical behaviour and can be described by the following four assumptions:

\begin{itemize}\addtolength{\itemindent}{1em}
	\item[$(A_1)$] $f \in C^1(\mathbb{R} \setminus \{0\}) \cap C^{0, \gamma}_{\text{loc}}(\mathbb{R})$ for some $\gamma \in (0,1)$.
	\item[$(A_2)$] There exist $q \in (p,p^*)$ and $C>0$ such that $|sf'(s)|, |f(s)| \leqslant C (|s|^{q-1}+1)$ for all $s \in \mathbb{R}\setminus\{0\}$.
	Here $p^*=\frac{Np}{N-p}$ if $p<N$ and $p^*=+\infty$ if $p\geqslant N$.
	\item[$(A_3)$] $f'(s) > (p-1)\dfrac{f(s)}{s} > 0$ for all $s \in \mathbb{R}\setminus \{0\}$, and $\limsup\limits_{s \to 0} \dfrac{f(s)}{|s|^{p-2}s} < \lambda_1(\Omega)$, where $\lambda_1(\Omega)$ is the first eigenvalue of the zero Dirichlet $p$-Laplacian in $\Omega$, that is, 
	\begin{equation*}
	\lambda_1(\Omega) = \min_{u \in \mathcal{S}} \int_\Omega |\nabla u|^p \, dx,
	\end{equation*}
	where $\mathcal{S} := \{u \in \wolp(\Omega): \|u\|_{L^p(\Omega)} = 1\}$.	
	\item[$(A_4)$] There exist $s_0 > 0$ and $\theta > p$ such that $0 < \theta F(s) \leqslant s f(s)$ for all $|s| > s_0$, where 
	$$
	F(s) := \int_0^s f(t) \, dt.
	$$
\end{itemize}
The model case of the nonlinearity which satisfies $(A_1)-(A_4)$ is $f(u) = C |u|^{q-2}u$ for any $q \in (p, p^*)$ and $C>0$.

\item[II.] The second type of the nonlinearity $f$ is the \textit{resonant} case  $f(u) = \lambda_2(\Omega) |u|^{p-2}u$, where $\lambda_2(\Omega)$ is the second eigenvalue of the zero Dirichlet $p$-Laplacian in $\Omega$ which can be characterized as (see \cite{drabekrobinson}) 
$$
\lambda_2(\Omega) = \inf_{\mathcal{A} \subset \mathcal{F}_2} \sup_{u \in \mathcal{A}} \int_\Omega |\nabla u|^p \, dx,
$$
where $\mathcal{F}_2 := \{\mathcal{A} \subset \mathcal{S}: \text{there exists a continuous odd surjection } h: S^1 \to \mathcal{A}\}$ and $S^1$ stands for a circle in $\mathbb{R}^2$. 
\end{enumerate}

Weak solutions of \eqref{D} are critical points of the energy functional $E: \wolp(\Omega) \to \mathbb{R}$ defined as
$$
E[u] = \frac{1}{p} \int_\Omega |\nabla u|^p \, dx - \int_\Omega F(u) \, dx.
$$
The functional $E$ is weakly lower semicontinuous and belongs to $C^1(\wolp(\Omega))$. 

\begin{remark}\label{rem:regularity}
	Let $u \in \wolp(\Omega)$ be a weak solution of \eqref{D}. Then $u \in L^\infty(\Omega)$ (it can be shown using a bootstrap argument, see, e.g., \cite[Lemma 3.14]{drabekkufner}), and hence $u \in C^{1,\delta}_{\text{loc}}(\Omega)$, see \cite{tolksdorf}. 
	Moreover, if $x \in \partial \Omega$ has a neighbourhood $B(x,\varepsilon)$ such that $\partial \Omega \cap B(x,\varepsilon)$ is of class $C^{1,\varsigma}$ (e.g., under the  assumption $(O_3)$), then $u \in C^{1,\kappa}_{\text{loc}}(\Omega \cup (\partial \Omega \cap B(x,\varepsilon)))$, as it follows from \cite[Section 3]{lieberman}.
\end{remark}

In this note we will study weak \textit{nodal} (or, equivalently, \textit{sign-changing}) solutions of \eqref{D}, i.e., solutions $u \in \wolp(\Omega)$ such that $u = u^+ + u^-$, where $u^+:= \max\{u, 0\}$, $u^-:= \min\{u, 0\}$, and $u^\pm \not\equiv 0$ in $\Omega$.
In the resonant case $f(u) = \lambda_2(\Omega) |u|^{p-2}u$, any nontrivial solution (which is necessarily nodal) can be naturally referred as \textit{second eigenfunction}. 
Noting that $u^\pm \in \wolp(\Omega)$, it is not hard to see that any nodal solution of \eqref{D} belongs to the \textit{nodal Nehari set} 
\begin{align*}\label{def:nodalNehari}
\mathcal{M}(\Omega) := \Bigl\{u \in \wolp(\Omega):\, u^\pm \not\equiv 0, ~ 
E'[u] u^\pm  \equiv \int_\Omega |\nabla u^\pm|^p \, dx - \int_\Omega u^\pm \, f(u^\pm) \, dx = 0 \,\Bigr\}.
\end{align*}

In the superlinear case (i.e., $f$ satisfies $(A_1)-(A_4)$), among the set of all nodal solutions of \eqref{D} we will be interested in \textit{least energy} nodal solutions, which can be obtained as follows. \textit{Any} minimizer of the problem
\begin{equation}\label{minimization_problem}
E[u] \to \mathrm{min}; 
\quad 
u \in {\mathcal M}(\Omega)
\end{equation}
is a least energy nodal solution of \eqref{D}, see \cite{cosio,BartschWethWillem,bobkol2016}.

An analogous property also holds in the resonant case $f(u) = \lambda_2(\Omega) |u|^{p-2}u$. Namely, for \textit{any} function $u \in \mathcal{M}(\Omega)$ there exist nonzero $\alpha, \beta \in \mathbb{R}$ such that $\alpha u^+ + \beta u^-$ is a second eigenfunction of \eqref{D}. Indeed, since in the resonant case the problem \eqref{D} is homogeneous, we can assume that $\|u^\pm\|_{L^p(\Omega)} = 1$. Therefore, $u \in \mathcal{M}(\Omega)$ reads as $\lambda_2(\Omega) = \int_\Omega |\nabla u^\pm|^p \, dx$.
Consider the set
$$
\mathcal{A} := \Bigl\{v \in \wolp(\Omega):~ v = \alpha u^+ + \beta u^-,~  |\alpha|^p  + |\beta|^p = 1 \Bigr\}.
$$
It is not hard to see that $\mathcal{A} \subset \mathcal{F}_2$ and $\lambda_2(\Omega) = \int_\Omega |\nabla v|^p \, dx$ for any $v \in \mathcal{A}$.
If we suppose now that there is no second eigenfunction which belongs to $\mathcal{A}$, then we apply the deformation lemma of Ghoussoub \cite[Lemma~3.7]{ghoussoub} (see also the particular statement for our case in \cite[Proposition~2.2]{anoopdrabeksasi}) to generate $\widetilde{\mathcal{A}} \subset \mathcal{F}_2$ such that
$$
\lambda_2(\Omega) \leqslant \sup_{v \in \widetilde{\mathcal{A}}} \int_\Omega |\nabla v|^p \, dx < \sup_{v \in \mathcal{A}} \int_\Omega |\nabla v|^p \, dx = \lambda_2(\Omega),
$$
a contradiction. Since any second eigenfunction changes its sign in $\Omega$, we conclude that $\alpha, \beta \neq 0$.

The above-mentioned facts indicate that least energy nodal solutions of \eqref{D} under the assumptions $(A_1)-(A_4)$ and second eigenfunctions of the zero Dirichlet $p$-Laplacian are conceptually the same objects. 
Analogous observation links ground states of \eqref{D} (i.e., least energy solutions) with first eigenfunctions. See \cite{GrumiauParini} for rigorous results in this direction.

\medskip

Let us define the \textit{nodal set} of $u \in C(\Omega)$ as
\begin{equation*}
\mathcal{Z}(u) = \overline{\{x \in \Omega:~ u(x) = 0\}}.
\end{equation*}
If $u$ is a nodal solution of \eqref{D}, then $\mathcal{Z}(u)$ is nonempty. Moreover, the strong maximum principle (see, for instance, \cite{vazquez}) implies that $\mathcal{Z}(u)$ does not have isolated points.
However, we are not aware of the unique continuation property for nodal solutions of \eqref{D}, i.e., the fact that $\text{Int}(\mathcal{Z}(u)) = \emptyset$.

Connected components of $\Omega \setminus \mathcal{Z}(u)$ are called nodal domains of $u$.
Note that each least energy nodal solution of \eqref{D} has exactly two nodal domains. Indeed, the arguments of \cite[p.~1051]{cosio} can be easily adopted for the superlinear case, and the resonant case was treated in \cite{CuestaNodal}. 

\medskip

In this note we intend to prove that the nodal set $\mathcal{Z}(u)$ intersects the boundary of $\Omega$. 
Namely, we prove the following result.
\begin{thm}\label{thm:1}
	Assume that $\Omega$ satisfies $(O_1)-(O_4)$, and either $f(u) = \lambda_2(\Omega) |u|^{p-2}u$ or $(A_1)-(A_4)$ holds. 
	Let $u$ be a least energy nodal solution of \eqref{D}. Then
	\begin{equation}\label{eq:dist0}
	\mathrm{dist}(\mathcal{Z}(u), \partial \Omega) = 0.
	\end{equation}
\end{thm}

In the \textit{resonant} case with the classical Laplace operator, such type of results is connected with the   \textit{Payne conjecture} \cite[p.~467]{payne} which stated that the nodal set of second eigenfunctions of the zero Dirichlet Laplacian in any domain cannot be closed. 
It is known that the conjecture is not generally true, see counterexamples constructed in  \cite{HoffmannOstenhof} and \cite{fournais} for special domains. However, determination of classes of domains for which the Payne conjecture is valid remains an attractive problem. 
The interested reader will easily find various results in this direction in the literature.
Let us specially emphasize the work \cite{payne1973} where Payne proved his conjecture for planar sets which satisfy $(O_1)$ and $(O_2)$. Under the same assumptions, the general higher-dimensional case was treated in \cite{damascelliNodal}. 
Being proved by entirely different arguments, our Theorem~\ref{thm:1} represents the direct generalization of these results  for the nonlinear settings under the additional assumptions $(O_3)$ and $(O_4)$. 
In the resonant case with the $p$-Laplacian, the validity of the Payne conjecture was proposed as an open problem even in the case of a ball, see \cite[Remark~4.2]{anoopdrabeksasi}. Easily, a ball satisfies $(O_1)-(O_4)$, and hence Theorem~\ref{thm:1} applies.

In the \textit{superlinear} case with the Laplace operator, the Payne conjecture was proved in \cite{aftalion} for a ball in $\mathbb{R}^N$ with $N \geqslant 2$ and for a two-dimensional annulus. 
Considering the homogeneous superlinear case $f(u) = |u|^{q-2} u$ and smooth $\Omega \subset \mathbb{R}^2$, the authors of \cite{GrumiauTroestler2009} were able to prove \eqref{eq:dist0} by appealing to \cite{allesandrini} when $q$ is sufficiently close to $2$ and $\Omega$ is convex; the case of large $q$ was considered in \cite{grossi}. See also \cite{Bonhuere2008} for related questions. 
However, we were not aware of the results on the Payne conjecture for the superlinear case with the general $p$-Laplace operator. 

\medskip

This note is organized as follows. In Section \ref{section:aux}, we give several necessary facts about polarization of functions and sets. In Section \ref{section:proof}, we prove Theorem \ref{thm:1}. The proof of Theorem~\ref{thm:1} consists of two steps. 
In the first step, we show that there \textit{always} exists a least energy nodal solution of \eqref{D} which satisfies \eqref{eq:dist0}. Such a weak form of the Payne conjecture is proved using $(O_1)$ and $(O_2)$ only. The arguments are based on the analysis of the behaviour of the polarization of least energy nodal solutions with respect to moving hyperplanes.
In the second step, we prove that if \eqref{eq:dist0} does not hold for some least energy nodal solution, then the least energy nodal solution obtained in the previous step has a contradictory behaviour on $\partial \Omega$. 
Here the arguments are based on the Hopf maximum principle which requires us to use $(O_3)$ and $(O_4)$.
This finally implies that \textit{any} least energy nodal solution of \eqref{D} satisfies \eqref{eq:dist0}.
In Section \ref{section:discussion}, we discuss possible relaxations of the assumptions $(A_1)-(A_4)$.

\section{Auxiliary facts}\label{section:aux}

First we recall the notion of \textit{polarization} (or, equivalently, \textit{two-point rearrangement}) of sets and functions, see, e.g., \cite{brocksol,BartschWethWillem}.
Consider the hyperplane $H_a := \{x \in \mathbb{R}^N: x_1 = a \}$ where $x := (x_1,x_2,\dots,x_N)$, and let $\sigma_a(x) := (2a - x_1, x_2, \dots, x_N)$ be the reflection of $x$ with respect to $H_a$. 
Denote the half-spaces separated by $H_a$ as 
\begin{equation*}
\Sigma_a^+ := \{ x \in \mathbb{R}^N:~ x_1 > a \} 
\quad \text{and} \quad 
\Sigma_a^- := \{ x \in \mathbb{R}^N:~ x_1 < a \}.
\end{equation*}
We define two polarizations of a measurable set $\Omega$ with respect to $H_a$ as follows:
\begin{equation*}
P_a \Omega = 
\begin{cases}
\Omega \cap \sigma_a(\Omega) &\text{in } \Sigma_a^+,\\
\Omega \cup \sigma_a(\Omega) &\text{in } \Sigma_a^-,\\
\Omega &\text{on } H_a,
\end{cases}
\qquad 
\widetilde{P}_a \Omega = 
\begin{cases}
\Omega \cup \sigma_a(\Omega) &\text{in } \Sigma_a^+,\\
\Omega \cap \sigma_a(\Omega) &\text{in } \Sigma_a^-,\\
\Omega &\text{on } H_a.
\end{cases}
\end{equation*}
Here $\widetilde{P}_a$ is introduced only for simplicity of further arguments, since $\widetilde{P}_a \Omega = \mathbb{R}^N \setminus (P_a (\mathbb{R}^N \setminus \Omega))$. 

It is not hard to see that $P_a$ and $\widetilde{P}_a$ satisfy the following domain monotonicity property.
\begin{lemma}\label{lemma_domain_monotonicity}
	If $\Omega_1 \subset \Omega_2$, then $P_a \Omega_1 \subset P_a \Omega_2$ and $\widetilde{P}_a \Omega_1 \subset \widetilde{P}_a \Omega_2$.
\end{lemma}

The following result links the Steiner symmetry of $\Omega$ with polarizations of $\Omega$, see \cite[Lemma~6.3]{brocksol}. 
We give its proof for the sake of completeness. 
\begin{lemma}\label{lemma_Q_polarized0}
	$\Omega$ satisfies $(O_2)$ if and only if $P_a \Omega = \Omega$ for all $a \geqslant 0$ and $\widetilde{P}_a \Omega = \Omega$ for all $a \leqslant 0$.                       
\end{lemma}
\begin{proof}
	Recall that $\Omega$ is Steiner symmetric with respect to the hyperplane $H_0$ if and only if for any point $(b_1,b_2,\dots,b_N) \in \Omega$ the whole line segment $\{(s,b_2,\dots,b_N) \in \mathbb{R}^N: |s| \leqslant |b_1|\}$ is a subset of $\Omega$.
	Since every segment of this kind is stable under $P_a$ for $a \geqslant 0$ and $\widetilde{P}_a$ for $a \leqslant 0$, we get the necessary part of the lemma.
	
	To prove the sufficient part we suppose, by contradiction, that there exist $b \in \Omega$ and $c \in \{(s,b_2,\dots,b_N) \in \mathbb{R}^N: |s| \leqslant |b_1|\}$ such that $c \not\in \Omega$. Assume, without loss of generality, that $b_1 > 0$, and take $a = (b_1 + c_1)/2 \geqslant 0$. 
	Then we see that $P_a$ will exchange the points $c$ and $b$, that is, $P_a \Omega \neq \Omega$. 
	It is a contradiction.
\end{proof}

We will use $P_a$ to polarize functions. 
Let $v$ be a measurable function defined on the whole $\reals^N$. 
Define the polarization of $v$ with respect to $H_a$ as
\begin{equation*}
(P_a v)(x) = 
\begin{cases}
\min \{v(x), v(\sigma_a(x))\}, &x \in \Sigma_a^+,\\
\max \{v(x), v(\sigma_a(x))\}, &x \in \Sigma_a^-,\\
v(x), &x \in H_a.
\end{cases}
\end{equation*}
It is known that $(P_a v)^\pm = P_a(v^\pm)$, see \cite[Lemma~2.1]{BartschWethWillem}. Therefore, 
\begin{equation}\label{eq:espantion_of_polarization_positive_negative}
P_a (v^+ + v^-) = P_a v = (P_a v)^+ + (P_a v)^- = P_a (v^+) + P_a (v^-),
\end{equation}
and we may write $P_a v^\pm$ for short.
It is not hard to see that $\mathrm{supp}\, P_a v^+ = P_a (\mathrm{supp}\, v^+)$ and $\mathrm{supp}\, P_a v^- = \widetilde{P}_a (\mathrm{supp}\, v^-)$. 
Therefore, we deduce from \eqref{eq:espantion_of_polarization_positive_negative} that 
\begin{equation}\label{eq:decomposition_of_supp}
\mathrm{supp}\, P_a v = P_a(\mathrm{supp}\, v^+) \cup \widetilde{P}_a(\mathrm{supp}\, v^-).
\end{equation}

The following result easily follows from Lemmas~\ref{lemma_domain_monotonicity} and \ref{lemma_Q_polarized0}.
\begin{cor}\label{lemma_Q_polarized}
	Let $\Omega$ satisfy $(O_2)$ and $v$ be such that $\mathrm{supp}\, v \subset \overline{\Omega}$. Then $\mathrm{supp}\, P_a v^+ \subset \overline{\Omega}$ for all $a \geqslant 0$ and $\mathrm{supp}\, P_a v^- \subset \overline{\Omega}$ for all $a \leqslant 0$.
\end{cor}

It is known that if $v \in \wlp(\mathbb{R}^N)$, then $P_a v \in \wlp(\mathbb{R}^N)$, see \cite[Lemma~5.3]{brocksol} applied to $v^\pm$. 
Moreover,
\begin{align*}
&\int_{\mathbb{R}^N} |\nabla P_a v^\pm|^p \,dx = \int_{\mathbb{R}^N} |\nabla v^\pm|^p \,dx, 
\\
&\int_{\mathbb{R}^N} P_a v^\pm f(P_a v^\pm) \,dx = \int_{\mathbb{R}^N} v^\pm f(v^\pm) \,dx, 
\quad \int_{\mathbb{R}^N} F(P_a v^\pm) \,dx = \int_{\mathbb{R}^N} F(v^\pm) \,dx,
\end{align*}
see \cite[Lemmas~2.2 and 2.3]{BartschWethWillem}. 
Furthermore, it is possible to polarize functions from $\wolp(\Omega)$, where $\Omega$ is an open set, by considering their trivial extension to $\mathbb{R}^N$.
Resulting functions belong to $\wolp(P_a \Omega \cup \widetilde{P}_a \Omega)$, as it follows from \cite[Corollary~5.1]{brocksol}.
This result can be clarified in the following way. 
Let $v \in \wolp(\Omega) \cap C(\Omega)$ and let $\mathcal{O}^+ := \{x \in \Omega: v(x) > 0\}$ and $\mathcal{O}^- := \{x \in \Omega: v(x) < 0\}$. Then \cite[Lemma~5.6]{CuestaFucik} implies that $v^\pm \in \wolp(\mathcal{O}^\pm)$, and therefore $P_a v \in \wolp(P_a \mathcal{O}^+ \cup \widetilde{P}_a \mathcal{O}^-)$. 
Recalling now \eqref{eq:decomposition_of_supp}, we arrive at the following fact. 
\begin{lemma}\label{lemma_polarized_nehari_function}
	Let $\Omega$ satisfy $(O_1)$. 
	Let $u \in \mathcal M(\Omega) \cap C(\Omega)$ and $\mathrm{supp}\, P_a u \subset \overline{\Omega}$ for some $a \in \mathbb{R}$.
	Then $P_a u \in \mathcal M(\Omega)$ and $E[P_a u] = E[u]$. 
	In particular, if $u$ is a minimizer of \eqref{minimization_problem} such that $\mathrm{supp}\, P_a u \subset \overline{\Omega}$, then $P_a u$ is also a minimizer of \eqref{minimization_problem}, that is, $u$ and $P_a u$ are both least energy nodal solutions of \eqref{D}.
\end{lemma}

\section{Proof of the main result}\label{section:proof}

Throughout this section $u$ denotes a least energy nodal solution of \eqref{D} and we suppose, by contradiction, that the nodal set of $u$ does not intersect the boundary of $\Omega$, i.e,
\begin{equation}\label{eq:dist1}
\mathrm{dist}(\mathcal{Z}(u), \partial \Omega) =: d > 0.
\end{equation}
We will assume, in particular, that $u > 0$ in $\{x \in \Omega: \text{dist}(x, \partial \Omega) < d\}$, and derive a contradiction.

Notice that if $u$ is negative near $\partial \Omega$, then $v := -u$ is a solutions of 
\begin{equation*}
\left\{
\begin{aligned}
-\Delta_p v &= -f(-v) 
&&\text{in } \Omega, \\
v &= 0  &&\text{on } \partial \Omega.
\end{aligned}
\right. 
\end{equation*}
Since $g(s) := -f(-s)$ satisfies $(A_1)-(A_4)$, we can apply the arguments from below to $v$. Alternatively, we can apply the following arguments to $u$ using $\widetilde{P}_a$ instead of $P_a$.

\medskip

The idea of the proof of Theorem \ref{thm:1} is based on the following simple observation. Considering the polarization $P_a u$, we continuously increase $a \geqslant 0$ until a moment $d_1/2$ (see below) when $\mathcal{Z}(P_{d_1/2} u)$ touches $\partial \Omega$ for the first time. From Lemma \ref{lemma_polarized_nehari_function} we see that $P_{d_1/2} u$ is a least energy nodal solution of \eqref{D}. However, we show that $P_{d_1/2} u$ cannot be a solution since it contradicts the Hopf maximum principle at a special point of the boundary. The details are as follows.

\medskip
As the first step of the proof of Theorem \ref{thm:1}, we show that apart from $u$, \eqref{D} possesses a least energy nodal solution $v$ such that $\mathrm{dist}(\mathcal{Z}(v), \partial \Omega) = 0$.

Consider the $e_1$-distance between $\mathcal{Z}(u)$ and the left part of $\partial \Omega$ (see Fig.\ \ref{fig1}):
$$
d_1 := \min\left\{\alpha > 0:~ (x_1, x_2,\dots,x_N) \in \mathcal{Z}(u)
~\text{ and }~
(x_1 - \alpha,x_2,\dots,x_N) \in \partial \Omega \right\}.
$$
We have $d_1 \geqslant d > 0$. 
Denote the subset of $\mathcal{Z}(u)$ which delivers the minimum to $d_1$ as $\mathcal{Y}(u)$, i.e., 
\begin{equation}\label{def:Y(u)}
\mathcal{Y}(u) := \left\{(x_1, x_2,\dots,x_N) \in \mathcal{Z}(u):~ (x_1 - d_1,x_2,\dots,x_N) \in \partial \Omega \right\}.
\end{equation}
Finally, we define the relative boundary of $\mathcal{Y}(u)$ as 
$$
\partial \mathcal{Y}(u) = \{x \in \mathcal{Y}(u): \text{ any neighbourhood of } x \text{ contains } y \in \mathcal{Z}(u) \setminus \mathcal{Y}(u)\}.
$$
It is not hard to see that $\mathcal{Y}(u)$ and $\partial \mathcal{Y}(u)$ are not empty. 

\begin{lemma}\label{lemma_moving_polarization}
	Let $(O_1)$ and $(O_2)$ be satisfied. 
	Then $\mathrm{supp}\, P_a u \subset \overline{\Omega}$ for all $a \in [0,d_1/2]$ and $\mathrm{dist}(\mathcal{Z}(P_{d_1/2} u), \partial \Omega) = 0$. 
	In particular, $P_a u$ is a least energy nodal solution of \eqref{D} for all $a \in [0,d_1/2]$.
\end{lemma}
\begin{proof}
	Recalling \eqref{eq:decomposition_of_supp}, we will study the behaviour of $\mathrm{supp}\, P_a u^+$ and $\mathrm{supp}\, P_a u^-$ with respect to $a$. First, Corollary~\ref{lemma_Q_polarized} implies that $\mathrm{supp}\, P_a u^+ \subset \overline{\Omega}$ for all $a \geqslant 0$. 
	Consider now $\mathrm{supp}\, P_a u^-$. 
	Take any $x \in \mathrm{supp}\, u^-$. Then, in view of \eqref{eq:dist1}, the assumption that $u > 0$ in $\{x \in \Omega: \text{dist}(x, \partial \Omega) < d\}$, and the definition of $d_1$, we see that $(x_1 - d_1,x_2,\dots,x_N) \in \overline{\Omega}$. Moreover, $(x_1 - d_1,x_2,\dots,x_N) \in \partial \Omega$ if and only if $x \in \mathcal{Y}(u)$. 
	Hence, from $(O_2)$ we get 
	\begin{align}
	\notag
	\sigma_{a}(x)	&= (2a - x_1,x_2,\dots,x_N) \in \overline{\Omega}, \quad \forall a \in [0, d_1/2],
	\\
	\label{eq:sigma_d_1/2}
	\sigma_{d_1/2}(x)	&= (d_1 - x_1,x_2,\dots,x_N) \in \partial \Omega \quad \iff \quad x \in \mathcal{Y}(u).
	\end{align}
	Consequently, $\widetilde{P}_a x \in \overline{\Omega}$ for all $a \in [0, d_1/2]$. 
	Since $x \in \mathrm{supp}\, u^-$ is arbitrary, we recall that $\mathrm{supp}\, P_a u^- = \widetilde{P}_a(\mathrm{supp}\, u^-)$ and hence conclude that $\mathrm{supp}\, P_a u^- \subset \overline{\Omega}$ for all $a \in [0, d_1/2]$. 
	
	Let us take some $x \in \partial \mathcal{Y}(u)$. 
	By definition, any small neighbourhood of $x$ contains $y \in \mathcal{Z}(u) \setminus \mathcal{Y}(u)$. Therefore, $\sigma_{d_1/2}(y) \in \Omega$. 
	On the one hand, since $u > 0$ near $\partial \Omega$, we see that $u(\sigma_{d_1/2}(y)) > 0$. On the other hand, $u(y) = 0$. Thus, $P_{d_1/2}$ will exchange $y$ and $\sigma_{d_1/2}(y)$, and hence $P_{d_1/2} u(\sigma_{d_1/2}(y)) = 0$. 
	Since a neighbourhood of $x$ is arbitrary, we conclude that $\mathrm{dist}(\mathcal{Z}(P_{d_1/2} u), \partial \Omega) = 0$.
	
	Finally, applying Lemma \ref{lemma_polarized_nehari_function}, we see that $P_a u$ is a least energy nodal solution of \eqref{D} for all $a \in [0,d_1/2]$.
\end{proof}

\begin{remark}
	Lemma~\ref{lemma_moving_polarization} is reminiscent of the moving plane method of Serrin \cite{serrin}, with the difference that, instead of the reflection, the polarization of functions is used.
\end{remark}

\begin{figure}[!h]
	\centering
	\includegraphics[width=0.7\linewidth]{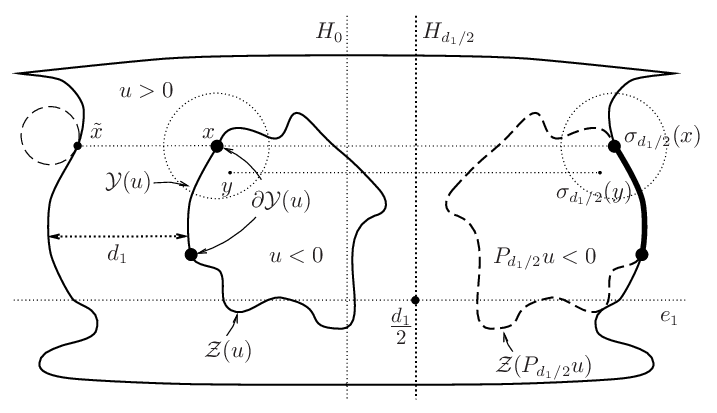}
	\caption{Nodal sets of $u$ and $P_{d_1/2} u$.}
	\label{fig1}
\end{figure}

\medskip
As the second step of the proof of Theorem \ref{thm:1}, we show that under the assumption \eqref{eq:dist1} the least energy nodal solution $P_{d_1/2} u$ has a contradictory behaviour on $\partial \Omega$. 
We start with the following auxiliary result.
 \begin{lemma}\label{lemma_contact_set_properties}
	Let $(O_1)$ and $(O_4)$ be satisfied, and let $x \in \mathcal{Y}(u)$. Then $|\nabla u(x)| \neq 0$. 
	Moreover, $\mathcal{Z}(u)$ is of class $C^2$ in a neighbourhood of $x$, and hence $\mathrm{supp}\,u^\pm$ satisfies the interior ball condition at $x$.
\end{lemma}
\begin{proof}
	Fix any $x \in \mathcal{Y}(u)$. Then $\tilde{x} := (x_1 - d_1, x_2, \dots, x_N) \in \partial \Omega$, see Fig.\ \ref{fig1}.
	Note that in view of $(O_4)$ we can find a ball $B(\tilde{y}, \varepsilon)$ which touches $\tilde{x}$ from the outside of $\Omega$. Since we assume that \eqref{eq:dist1} is satisfied, the translated ball $B(\tilde{y} + d_1 e_1, \varepsilon)$ is a subset of $\Omega$ for $\varepsilon>0$ small enough. 
	Moreover, from the definition of $d_1$ it is not hard to deduce that the $e_1$-distance between any $z \in B(\tilde{y}, \varepsilon)$ and $\mathcal{Z}(u)$ is greater than or equal to $d_1$, and it is equal to $d_1$ if and only if $z \in \partial \Omega$ and $z+d_1e_1 \in \mathcal{Y}(u)$. 
	Since we assume that $u>0$ in the neighbourhood of $\partial \Omega$, we conclude that $z+d_1 e_1 \in \mathrm{supp}\,u^+$ and hence $B(\tilde{y} + d_1 e_1, \varepsilon) \subset \mathrm{supp}\,u^+$. 
	
	Recalling now that $u \in C^{1,\delta}_{\text{loc}}(\Omega)$ (see Remark \ref{rem:regularity}), we use the Hopf maximum principle \cite[Theorem~5]{vazquez} to get $|\nabla u(x)| \neq 0$.
	This implies that the $p$-Laplacian is strictly elliptic in a neighbourhood of $x$, which yields that $u$ is $C^2$-smooth in this neighbourhood. Applying the implicit function theorem, we conclude that $\mathcal{Z}(u)$ is a graph of a $C^{2}$-function in a neighbourhood of $x$, and the interior ball condition on $\mathrm{supp}\,u^\pm$ at $x$ follows immediately.
\end{proof}

\begin{lemma}\label{lemma_contradiction}
	Let $(O_1)-(O_4)$ be satisfied. Then $P_{d_1/2} u$ is not a solution of \eqref{D}.
\end{lemma}
\begin{proof}
Let us denote, for simplicity, $v := P_{d_1/2} u$. 
Take any $x \in \partial \mathcal{Y}(u)$. By the definition of $\partial \mathcal{Y}(u)$, we have $\sigma_{d_1/2}(x) \in \partial \Omega$.
Fix a small $\varepsilon \in (0,d)$ and consider the neighbourhoods $B(x,\varepsilon)$ and $B(\sigma_{d_1/2}(x),\varepsilon)$, see Fig.~\ref{fig1}.
By Lemma \ref{lemma_contact_set_properties}, the first neighbourhood has nonempty intersection with $\mathrm{supp}\,u^-$, while the second does not (since we assume that $u>0$ in the neighbourhood of $\partial \Omega$). 
Hence, considering any $y \in B(x,\varepsilon) \cap \mathrm{supp}\,u^-$, we obtain that $v(y) = u(\sigma_{d_1/2}(y)) > 0$ and $v(\sigma_{d_1/2}(y)) = u(y) < 0$. Thus, small neighbourhoods of $\sigma_{d_1/2}(x)$ intersect with $\mathrm{supp}\,v^-$. 
Moreover, since $x \in \partial \mathcal{Y}(u)$, we also see that small neighbourhoods of $\sigma_{d_1/2}(x)$ intersect with $\mathrm{supp}\,v^+$ and 
\begin{equation}\label{eq:nonempty_intersection_1}
\partial \Omega \cap \mathrm{supp}\,v^+ \cap B(\sigma_{d_1/2}(x),\varepsilon) \neq \emptyset.
\end{equation}

By Lemma \ref{lemma_contact_set_properties}, $\mathrm{supp}\,u^-$ satisfies the interior ball condition at $x$, which implies that $\mathrm{supp}\,v^-$ and, consequently, $\Omega$ satisfy the interior ball condition at $\sigma_{d_1/2}(x)$.
Therefore, due to $(O_3)$, we get $v \in C^{1,\kappa}_{\text{loc}}(\Omega \cup (\partial \Omega \cap B(\sigma_{d_1/2}(x),\varepsilon)))$ (see Remark \ref{rem:regularity}), and consequently the Hopf maximum principle implies
\begin{equation}\label{eq:hopf>}
	\frac {\partial v}{\partial n}(\sigma_{d_1/2}(x)) > 0,
\end{equation}
where $n$ is the outward unit normal to $\partial \Omega$. 
On the other hand, since \eqref{eq:nonempty_intersection_1} holds, we get for any sufficiently small $\varepsilon > 0$ that 
\begin{equation}\label{eq:hopf<}
	\frac {\partial v}{\partial n}(\hat{x}) \leqslant 0, 
	\quad 
	\forall \hat{x} \in  \partial \Omega \cap \mathrm{supp}\,v^+ \cap B(\sigma_{d_1/2}(x),\varepsilon).
\end{equation}
Finally, taking a sequence $\{\hat{x}_k\}_{k \in \mathbb{N}} \subset \partial \Omega \cap \mathrm{supp}\,v^+ \cap B(\sigma_{d_1/2}(x),\varepsilon)$ such that $\hat{x}_k \to \sigma_{d_1/2}(x)$ as $k \to +\infty$, we obtain a contradiction between \eqref{eq:hopf>} and \eqref{eq:hopf<}, since $v$ is regular up to $\partial \Omega \cap B(\sigma_{d_1/2}(x),\varepsilon)$.
\end{proof}

\begin{remark}
	The idea of a contradiction was inspired by the proof of \cite[Lemma 3.2]{anoopdrabeksasi}, where the authors shown the nonradiality of second eigenfunctions of the zero Dirichlet $p$-Laplacian in a ball. 
\end{remark}

Finally, we see that Lemmas \ref{lemma_moving_polarization} and \ref{lemma_contradiction} contradict each other, which implies that \eqref{eq:dist1} does not hold, and hence Theorem \ref{thm:1} is valid.

\section{Discussion}\label{section:discussion}
The assumptions $(A_1)-(A_4)$ can be relaxed. Indeed, all we need for the proof of Theorem \ref{thm:1} is that: 1) a minimizer of the variational problem \eqref{minimization_problem} exists; 2) any minimizer of \eqref{minimization_problem} is a nodal solution of \eqref{D}; 3) such solutions are sufficiently regular and satisfy the Hopf maximum principle on smooth parts of $\partial \Omega$; 4) polarizations $P_a$ and $\widetilde P_a$ preserve all functionals present in $E$ and $\mathcal{M}(\Omega)$. For example, the assumption $(A_2)$ is used to guarantee the existence part; i.e., since $(A_2)$ holds, the embedding $\wolp(\Omega) \hookrightarrow L^q(\Omega)$ is compact. If we consider the embedding of $\wolp(\Omega)$ into suitable Orlicz spaces, we can impose an extension of $(A_2)$ which allows an exponential growth of $f$ for $p=N$ and arbitrary growth for $p > N$, see, e.g., \cite[(F4) and Lemma 5.6]{Nazar}. The assumption $(A_3)$ can be relaxed to allow, for instance, convex-concave nonlinearities. In this case, the variational problem \eqref{minimization_problem} has to be restricted to suitable subsets of $\mathcal{M}(\Omega)$ (see, e.g., \cite{bobkov}) and then the arguments from the present note can be applied.

\bigskip
\noindent
{\bf Acknowledgements.}
The first author was supported by the project LO1506 of the Czech Ministry of Education, Youth and Sports.
The second author wishes to thank the University of West Bohemia, where this research was started, for the invitation and hospitality.


\addcontentsline{toc}{section}{\refname}
\small
\begin{thebibliography}{99}
	
	\bibitem{allesandrini}
	Alessandrini, G. (1994). Nodal lines of eigenfunctions of the fixed membrane problem in general convex domains. Commentarii Mathematici Helvetici, 69(1), 142-154.
	\href{http://dx.doi.org/10.1007/BF02564478}{\nolinkurl{DOI:10.1007/BF02564478}}

	\bibitem{aftalion}
	Aftalion, A., \& Pacella, F. (2004). Qualitative properties of nodal solutions of semilinear elliptic equations in radially symmetric domains. Comptes Rendus Mathematique, 339(5), 339-344.	
	\href{http://dx.doi.org/10.1016/j.crma.2004.07.004}{\nolinkurl{DOI:10.1016/j.crma.2004.07.004}}

	\bibitem{anoopdrabeksasi}
	Anoop, T. V., Dr\'abek, P., \& Sasi, S. (2016). On the structure of the second eigenfunctions of the $p$-Laplacian on a ball. Proceedings of the American Mathematical Society, 144 (6), 2503-2512.
	\href{http://dx.doi.org/10.1090/proc/12902}{\nolinkurl{DOI:10.1090/proc/12902}}

	\bibitem{BartschWethWillem} 
	Bartsch, T., Weth, T., \& Willem, M. (2005). Partial symmetry of least energy nodal solutions to some variational problems. Journal d'Analyse Math\'ematique, 96(1), 1-18.
	\href{http://dx.doi.org/10.1007/bf02787822}{\nolinkurl{DOI:10.1007/bf02787822}}

	\bibitem{bobkol2016}
	Bobkov, V., \& Kolonitskii, S. (2017). On qualitative properties of solutions for elliptic problems with the $p$-Laplacian through domain perturbations. arXiv preprint arXiv:\href{https://arxiv.org/abs/1701.07408}{\nolinkurl{1701.07408}}.
	
	\bibitem{bobkov}
	Bobkov, V. E. (2013). On existence of nodal solution to elliptic equations with convex-concave nonlinearities. Ufa Mathematical Journal, 5(2), 18-30.
	\href{http://dx.doi.org/10.13108/2013-5-2-18}{\nolinkurl{DOI:10.13108/2013-5-2-18}}
	
	\bibitem{Bonhuere2008}
	Bonheure, D., Bouchez, V., Grumiau, C., \& Van Schaftingen, J. (2008). Asymptotics and symmetries of least energy nodal solutions of Lane–Emden problems with slow growth. Communications in Contemporary Mathematics, 10(04), 609-631.
	\href{http://dx.doi.org/10.1142/S0219199708002910}{\nolinkurl{DOI:10.1142/S0219199708002910}}

	\bibitem{brocksol}	
	Brock, F., \& Solynin, A. (2000). An approach to symmetrization via polarization. Transactions of the American Mathematical Society, 352(4), 1759-1796.
	\href{http://dx.doi.org/10.1090/s0002-9947-99-02558-1}{\nolinkurl{DOI:10.1090/s0002-9947-99-02558-1}}
	
	\bibitem{cosio}
	Castro, A., Cossio, J., \& Neuberger, J. M. (1997). A sign-changing solution for a superlinear Dirichlet problem. Rocky Mountain Journal of Mathematics, 27(4), 1041-1053.
	\href{http://dx.doi.org/10.1216/rmjm/1181071858}{\nolinkurl{DOI:10.1216/rmjm/1181071858}}
	
	\bibitem{CuestaFucik}
	Cuesta, M., De Figueiredo, D., \& Gossez, J. P. (1999). The beginning of the Fu\v{c}ik spectrum for the $p$-Laplacian. Journal of Differential Equations, 159(1), 212-238.
	\href{http://dx.doi.org/10.1006/jdeq.1999.3645}{\nolinkurl{DOI:10.1006/jdeq.1999.3645}}
	
	\bibitem{CuestaNodal}
	Cuesta, M., De Figueiredo, D. G., \& Gossez, J. P. (2000). A nodal domain property for the $p$-Laplacian. Comptes Rendus de l'Acad\'emie des Sciences-Series I-Mathematics, 330(8), 669-673.
	\href{http://dx.doi.org/10.1016/S0764-4442(00)00245-7}{\nolinkurl{DOI:10.1016/S0764-4442(00)00245-7}}
	
	\bibitem{damascelliNodal}
	Damascelli, L. (2000). On the nodal set of the second eigenfunction of the laplacian in symmetric domains in $\mathbb{R}^{N}$. Atti della Accademia Nazionale dei Lincei. Classe di Scienze Fisiche, Matematiche e Naturali. Rendiconti Lincei. Matematica e Applicazioni, 11(3), 175-181.
	\url{http://eudml.org/doc/252373}
	
	\bibitem{drabekkufner}
	Dr\'abek, P., Kufner, A., \& Nicolosi, F. (1997). Quasilinear elliptic equations with degenerations and singularities (Vol. 5). Walter de Gruyter.
	\href{http://dx.doi.org/10.1515/9783110804775}{\nolinkurl{DOI:10.1515/9783110804775}}

	\bibitem{drabekrobinson}
	Dr\'abek, P., \& Robinson, S. B. (1999). Resonance problems for the $p$-Laplacian. Journal of Functional Analysis, 169(1), 189-200.
	\href{http://dx.doi.org/10.1006/jfan.1999.3501}{\nolinkurl{DOI:10.1006/jfan.1999.3501}}
	
	\bibitem{fournais}
	Fournais, S. (2001). The nodal surface of the second eigenfunction of the Laplacian in $\mathbb{R}^D$ can be closed. Journal of Differential Equations, 173(1), 145-159.
	\href{http://dx.doi.org/10.1006/jdeq.2000.3868}{\nolinkurl{DOI:10.1006/jdeq.2000.3868}}

	\bibitem{ghoussoub}
	Ghoussoub, N. (1993). Duality and perturbation methods in critical point theory (Vol. 107). Cambridge University Press.	
	
	\bibitem{grossi}
	Grossi, M., Grumiau, C., \& Pacella, F. (2013). Lane–Emden problems: Asymptotic behavior of low energy nodal solutions. Annales de l'Institut Henri Poincare (C) Non Linear Analysis, 30(1), 121-140.
	\href{http://dx.doi.org/10.1016/j.anihpc.2012.06.005}{\nolinkurl{DOI:10.1016/j.anihpc.2012.06.005}}

	\bibitem{GrumiauParini}
	Grumiau, C., \& Parini, E. (2008). On the asymptotics of solutions of the Lane-Emden problem for the $p$-Laplacian. Archiv der Mathematik, 91(4), 354-365.
	\href{http://dx.doi.org/10.1007/s00013-008-2854-y}{\nolinkurl{DOI:10.1007/s00013-008-2854-y}}
	
	\bibitem{GrumiauTroestler2009}
	Grumiau, C., \& Troestler, C. (2009). Nodal line structure of least energy nodal solutions for Lane–Emden problems. Comptes Rendus Mathematique, 347(13-14), 767-771.
	\href{http://dx.doi.org/10.1016/j.crma.2009.04.023}{\nolinkurl{DOI:10.1016/j.crma.2009.04.023}}
	
	\bibitem{HoffmannOstenhof}
	Hoffmann-Ostenhof, M., Hoffmann-Ostenhof, T., \& Nadirashvili, N. (1997). The nodal line of the second eigenfunction of the Laplacian in $\mathbb{R}^2$ can be closed. Duke Mathematical Journal, 90(3), 631-640.
	\href{http://dx.doi.org/10.1215/S0012-7094-97-09017-7}{\nolinkurl{DOI:10.1215/S0012-7094-97-09017-7}}

	\bibitem{kawohl}
	Kawohl, B. (1985). Rearrangements and convexity of level sets in PDE. Lecture notes in mathematics, (1150), Springer Berlin Heidelberg.
	\href{http://dx.doi.org/10.1007/BFb0075060}{\nolinkurl{DOI:10.1007/BFb0075060}}
	
	\bibitem{lieberman}
	Lieberman, G. M. (1988). Boundary regularity for solutions of degenerate elliptic equations. Nonlinear Analysis: Theory, Methods \& Applications, 12(11), 1203-1219. 
	\href{http://dx.doi.org/10.1016/0362-546x(88)90053-3}{\nolinkurl{DOI:10.1016/0362-546x(88)90053-3}}
	
	\bibitem{Nazar}
	Nazarov, A. I. (2004). On Solutions to the Dirichlet problem for an equation with $p$-Laplacian in a spherical layer. Proceedings of the St. Petersburg Mathematical Society, 10, 33-62.
	\href{http://dx.doi.org/10.1090/trans2/214/03}{\nolinkurl{DOI:10.1090/trans2/214/03}}

	\bibitem{payne}
	Payne, L. E. (1967). Isoperimetric inequalities and their applications. SIAM Review, 9(3), 453-488.
	\href{http://dx.doi.org/10.1137/1009070}{\nolinkurl{DOI:10.1137/1009070}}
	
	\bibitem{payne1973}
	Payne, L. E. (1973). On two conjectures in the fixed membrane eigenvalue problem. Zeitschrift f\"ur angewandte Mathematik und Physik, 24(5), 721-729.
	\href{http://dx.doi.org/10.1007/BF01597076}{\nolinkurl{DOI:10.1007/BF01597076}}
	
	\bibitem{serrin}
	Serrin, J. (1971). A symmetry problem in potential theory. Archive for Rational Mechanics and Analysis, 43(4), 304-318.
	\href{http://dx.doi.org/10.1007/BF00250468}{\nolinkurl{DOI:10.1007/BF00250468}}
	
	\bibitem{tolksdorf}
	Tolksdorf, P. (1984). Regularity for a more general class of quasilinear elliptic equations. Journal of Differential equations, 51(1), 126-150.
	\href{http://dx.doi.org/10.1016/0022-0396(84)90105-0}{\nolinkurl{DOI:10.1016/0022-0396(84)90105-0}}
	
	\bibitem{vazquez}
	V\'azquez, J. L. (1984). A strong maximum principle for some quasilinear elliptic equations. Applied Mathematics and Optimization, 12(1), 191-202.
	\href{http://dx.doi.org/10.1007/bf01449041}{\nolinkurl{DOI:10.1007/bf01449041}}
	
\end{thebibliography}
\end{document}